\documentclass{amsart}
\usepackage{a4}
\usepackage{color}
\usepackage{float}
\usepackage[dvipsnames]{xcolor}
\usepackage{amsthm}
\usepackage{amsmath}
\usepackage{amssymb} 
\usepackage{paralist}
\usepackage{cite}
\theoremstyle{definition}
\newtheorem{Definition}{Definition}[section]
\newtheorem{Theorem}[Definition]{Theorem}

\newtheorem{Lemma}[Definition]{Lemma}
\newtheorem{Remark}[Definition]{Remark}
\newtheorem{Proposition}[Definition]{Proposition}

\title{Linear orbits of alternating forms on real vector spaces}
\author{Leonid Ryvkin}
\address{Leonid Ryvkin, Fakult\"at f\"ur Mathematik, Ruhr-Universit\"at Bochum, Universit\"atsstr. 150, 44801 Bochum, Germany}
\email{Leonid.Ryvkin@ruhr-uni-bochum.de}

\begin{document}

\maketitle

\begin{abstract}
In this note we complete the calculation of the number of $GL(\mathbb R^n)$-orbits on $\Lambda^k(\mathbb R^n)^*$, by treating the cases $(n,k)= (7,4)$ and $(8,5)$ not covered in the literature. We also calculate the number of of non-degenerate and stable orbits, as they are of special interest to multisymplectic and special geometry.   
\end{abstract}

\noindent\rule[0.5ex]{\linewidth}{1pt}

\section*{Introduction}
The last decades have seen a growing interest in multisymplectic geometry (cf., e.g., \cite{MR1694063, FRZ, MR2892558, MR2882772} for research in this area and \cite{MR916718, MR1871001} for the r\^ole of non-degenerate three-forms in special geometries). A manifold $M$ together with a differential form $\alpha$ of degree $k$ is called multisymplectic if $\alpha$ is closed and non-degenerate. The latter condition means that the map  
\[T_pM\to \Lambda^{k-1}(T_pM)^*,~ v\mapsto \iota_v(\alpha_p)\]
is injective for all $p\in M$, where $\iota_v(\alpha_p)$ denotes the contraction of a vector $v$ into the $k$-form $\alpha_p$ on $T_pM$.\\

In contrast to the symplectic case (i.e., the case $k=2$) already the linear theory is not trivial for general $k$. In fact, there are typically many different equivalence classes even for non-degenerate alternating $k$-forms on a real vector space $V$.

Since the knowledge of these linear equivalence classes (or linear types) are fundamental for multisymplectic geometry, we give here a complete answer to the following question:

\begin{quote}\emph{
Does the natural $GL(V)$-action on $\Lambda^k(V^*)$ have a finite or infinite number of orbits and what is this number in the former case?}
\end{quote}

Furthermore, we answer these questions  also for non-degenerate and stable orbits (i.e. orbits of non-degenerate and stable forms), as well. We analyse the cases $(k,n)=(4,7)$ and $(5,8)$ that were hitherto not covered by the literature and resume  all cases in Theorem \ref{thm} with proofs for the two cases mentioned above (Proposition \ref{85} and Proposition \ref{74}) and references to proofs for all other cases.\\

{\bf  Acknowledgements.}  The author thanks Tilmann Wurzbacher for suggesting the problem considered in this article and for useful discussions.

\section{Preliminaries}
Throughout this note, let  $V$ be a finite-dimensional real vector space of dimension $n>1$ and $k\in\{0,...,n\}$ a natural number.

\begin{Definition} We define the \emph{natural left action of $GL(V)$ on $\Lambda^k(V)$ (resp. the natural right action of $GL(V)$ on  $\Lambda^k(V^*)$)} by 
\begin{align*}&\theta_k:GL(V)\times \Lambda^k(V)\to \Lambda^k(V),&~~&\theta_{k,g}\left(\sum_iv_1^i\wedge ...\wedge v_k^i\right)=\sum_i(gv_1^i)\wedge ...\wedge (gv_k^i)\\
 &\theta^k:GL(V^*)\times \Lambda^k(V^*)\to \Lambda^k(V^*), &~~&\theta^k_g(\alpha)=\alpha\circ \theta_{k,g}.
\end{align*}
For $g\in GL(V)$, $\xi\in \Lambda^k(V)$ and $\alpha\in\Lambda^k(V^*)$, we will write $g\cdot \alpha$ for $\theta_{k,g}(\xi)$ and $g\cdot \alpha$ for $\theta^k_g(\alpha)$.
\end{Definition}

We will be especially interested in orbits of non-degenerate forms and stable orbits in $\Lambda^k(V^*)$, as they play an important role in multisymplectic geometry.
\begin{Definition}$ $
\begin{enumerate}
\item We define the contraction map $V\times \Lambda^k(V^*)\to \Lambda^{k-1}(V^*)$ by $(v,\alpha)\mapsto \iota_v\alpha$, where $(\iota_v\alpha)(v_1,...,v_{k-1})=\alpha(v,v_1,...,v_{k-1})$.
\item We uniquely extend the contraction to a linear map $\Lambda^l (V)\times \Lambda^k(V^*)\to \Lambda^{k-l}( V^* )$ for any $l$, such that $\iota_{\xi\wedge\zeta}\alpha=\iota_\zeta(\iota_\xi\alpha)$ holds for any $\xi\in\Lambda^l(V)$ and $\zeta\in\Lambda^{l'}V$.
\item An element $\alpha\in \Lambda^k(V^*)$ is called \emph{non-degenerate} if the map $V\to \Lambda^{k-1}(V^*),$ $v\mapsto \iota_v\alpha$ is injective and \emph{degenerate} otherwise. 
\item A multivector $\xi\in\Lambda^k(V)$ resp. a form $\alpha\in \Lambda^k(V^*)$ is called \emph{stable} if the orbit $GL(V)\cdot \xi\subset\Lambda^k(V)$ resp. $GL(V)\cdot \alpha\subset \Lambda^k(V^*)$ is open and \emph{unstable} otherwise.
\end{enumerate}
\end{Definition}

\begin{Remark}
Whenever at least one non-degenerate element exists in $\Lambda^k(V^*)$, all degenerate elements are automatically unstable.
\end{Remark}

As stability and non-degeneracy are preserved by the $GL(V)$-action, we will call an orbit $GL(V)\cdot \alpha\subset \Lambda^k(V^*)$ \emph{non-degenerate} (resp. \emph{stable}) if $\alpha$ is.

\begin{Remark}\label{hodgeremark}
The choice of a linear isomorphims from $V$ to $V^*$ induces an isomorphism $\Lambda^k(V)\to \Lambda^k(V^*)$, which maps orbits of $\theta_k$ to orbits of $\theta^k$. Thus the classification of orbits of $\theta_k$ is equivalent to the classification of orbits of $\theta^k$. In terms of $\Lambda^k(V)$ degeneracy of a multivector $\xi$ is equivalent to the existence of a subspace $W\subsetneqq V$ such that $\xi\in \Lambda^k(W)\subset \Lambda^k(V)$. We choose the perspective of orbits in $\Lambda^k(V^*)$, as it has a more immediate relation to multisymplectic geometry. 
\end{Remark}

\section{Statement of the theorem}

\begin{Theorem}\label{thm} Let $V$ be a real vector space of dimension $n> 1$ and $k\in \mathbb N$. The natural action of $GL(V)$ on $\Lambda^k(V^*)$ has the following number of orbits (resp. non-degenerate orbits and stable orbits) depending on $n$ and $k$.
\begin{table}[H]
\centering
\begin{tabular}{|l|l|l|l|}
\hline
                             & orbits      & non-deg. orbits& stable orbits \\ \hline
$k=1$                        & 2                           & 0                     &   1             \\ \hline
$k=2$, $n$ even                        & $\frac{n}{2}+1$ & 1  &   1  \\ \hline
$k=2$, $n$ odd                        & $\frac{n+1}{2}$ & 0 &   1  \\ \hline
$k=3$, $n=6$                   & 6                           & 3        &  2                           \\ \hline
$k=3$, $n=7$                   & 14                          & 8         &    2                        \\ \hline
$k=3$, $n=8$                   & 35                          & 21         &     3                      \\ \hline
$k=4$, $n=7$                   &   {20}      &{15}       & 4            \\ \hline
$k=4, n=8$                    & $\infty$                    & $\infty$                    &  0        \\ \hline
$k=5$, $n=8$                  &   {35}      &{31}            &    3    \\ \hline
$3\leq k\leq n-3$, $n\geq 9$ & $\infty$                    & $\infty$                       &     0  \\ \hline
$k=n-2$, $n \geq 6$, $n=2 \mod 4$   &   $\frac{n}{2}+2$             &      $\frac{n}{2}$    & 2                       \\\hline
$k=n-2$, $n \geq 5$, $n\neq 2 \mod 4$   &   $\lfloor\frac{n}{2}\rfloor+1$             &       $\lfloor\frac{n}{2}\rfloor - 1$      &      1                \\ \hline
$k=n-1$                      & 2                           & 0                      &  1             \\ \hline
$k=n$                        & 2                           & 1                    &        1         \\ \hline
\end{tabular}
\end{table}
\end{Theorem}

\begin{proof}
The cases $k\in \{0,1,2,n{-}2,n{-}1,n\}$ and the cases with infinite number of orbits have been solved in \cite{MR0286119}. The remaining cases for $k=3$ have been completed in \cite{MR691457} together with the stability considerations in \cite{MR2433628}. We settle the remaining cases $(n,k)=(7,4)$ in Proposition \ref{74} and $(n,k)=(8,5)$ in Proposition \ref{85}.
\end{proof}

\begin{Remark}
If two $k$-forms on $V$ lie in the same $GL(V)$-orbit, one says that they are ``equivalent'' or they have the same ``linear type''.
\end{Remark}

\section{Preparations}
 
In the sequel we will use the following technique which was used in \cite{MR0286119} to identify the orbits of $(n{-}2)$-forms in an $n$-dimensional vector space.
\begin{Lemma}\label{dualityLemma}
Let $\Omega\in \Lambda^n(V^*)$ be a volume form and $k\in\{0,...,n\}$. Then the linear isomorphism $c:\Lambda^k(V)\to \Lambda^{n-k}(V^*),~ \xi\mapsto\iota_\xi\Omega$ has the following properties:
\begin{enumerate}[(1)]
\item The isomorphism $c$ preserves stability.
\item For all $g\in GL(V)$ we have $g\cdot c(\xi)=det (g)c(g^{-1}\xi)$.
\item If $\xi$ lies in the orbit of $\zeta$, then $c(\xi)$ lies in the orbit of $c(\zeta)$ or in the orbit of $-c(\zeta)$. When $n{-}k$ is odd those two orbits coincide.
\item If $c(\xi)$ and $c(\zeta)$ lie in the same orbit, then $\xi$ lies in the orbit of $\zeta$ or in the orbit of $-\zeta$. When $k$ is odd those two orbits coincide.
\end{enumerate}
\end{Lemma}

\begin{proof}$ $
\begin{enumerate}[(1)]
\item This follows directly from the continuity of linear maps between finite-dimensional vector spaces.
\item This follows from the formula $g\cdot (\iota_\xi\Omega)=\iota_{(g^{-1}\cdot \xi)}g\Omega$ and the fact that $g\cdot \Omega=\theta^n_g(\Omega)=det(g)\Omega$ by the definition of the determinant.
\item If $\xi=g\cdot \zeta$, then $c(\zeta)=\frac{1}{det(g)}g\cdot c(\xi)$, so $c(\xi)$ lies in the orbit of $det(g) c(\zeta)$. By applying the endomorphism $\sqrt[n-k]{|det(g^{-1})|}\cdot id_V$ we see that $c(\xi)$ lies in the orbit of $\text{sign}(det(g))c(\zeta)$. The second part of the statement follows from the identity $(\lambda \cdot \text{id}_V)\cdot \alpha=\lambda^{n-k}\alpha$ for $\alpha\in \Lambda^{n-k}(V^*)$ and $\lambda\in \mathbb{R}^*$.
\item The proof of this part is analogous to that of part (3).
\end{enumerate}
\end{proof}

The following lemma will help us deduce the number of non-degenerate orbits, once we know the total number of orbits.
\begin{Lemma}\label{degenerateLemma}
The degenerate orbits of $k$-forms on a $n$-dimensional vector space are in one-to-one correspondance to all orbits of $k$-forms on a $n{-}1$-dimensional vector space.
\end{Lemma}
\begin{proof}
Let $W$ be a $(n{-}1)$-dimensional subspace of $V$ and $a\in V\backslash W$. Then $V=W\oplus \mathbb R\cdot a$. Let $\text{pr}_W:V\to W$ be the projection. We show that the $GL(W)$-orbits of $\Lambda^k(W^*)$ are in one-to-one correspondence to $GL(V)$-orbits of $\Lambda^k_{deg}(V^*)$, the degenerate $k$-forms on $V$, via the map 
\[\bar \phi:\frac{\Lambda^k(W^*)}{GL(W)}\to \frac{\Lambda^k_{deg}(V^*)}{GL(V)}, ~ [\alpha]=GL(W)\cdot \alpha\mapsto GL(V)\cdot \phi(\alpha)=[\phi(\alpha)],\]
 induced by 
\[\phi:\Lambda^k(W^*)\to \Lambda^k(W^*)\oplus(\Lambda^{k-1}(W^*)\otimes \mathbb( R\cdot a)^*)\cong\Lambda^k(V^*),~~\alpha\mapsto (\alpha \oplus 0).\] 

{\bf Surjectivity of $\bar\phi$.} Assume $\alpha\in \Lambda^k(V^*)$ is degenerate, i.e. there exists a $v$ in $V\backslash \{0\}$ such that $\iota_v\alpha=0$. Choosing a $g\in GL(V)$ such that $g^{-1}v=a$, we get $\iota_a(g\cdot \alpha)=0$, i.e. $(g\cdot \alpha)|_W\in \Lambda^k(W^*)$ satisfies $\phi((g\cdot \alpha)|_W)=g\cdot \alpha$. Consequently $\bar\phi([(g\cdot \alpha)|_W])=[\alpha]$ holds and the map $\bar \phi$ is surjective.\\

{\bf Injectivity of $\bar\phi$.} Let $\alpha,\beta\in \Lambda^k(W^*)$ and $\phi(\alpha)$ be equivalent to $\phi(\beta)$, i.e. there exists a map $g\in GL(V)$ such that $g\cdot \phi(\alpha)=\phi(\beta)$. Especially $ga$ and $a$ are both elements of the vector subspace $\text{ann}(\phi(\alpha)):=\{v\in V| \iota_v\phi(\alpha)=0\}$. We pick an element $h\in GL(V)$ such that $h(ga)=a$ and $h\cdot \phi(\alpha) =\phi(\alpha)$.
Then we have $(hg)\cdot\phi(\alpha)=g\cdot(h\cdot\phi(\alpha))=g\cdot \phi(\alpha)=\phi(\beta)$, where $hg(a)=a$. Consequently $\text{pr}_W\circ (hg)|_W$ yields a well-defined automorphism of $W$ satisfying $(\text{pr}_W\circ (hg)|_W)\cdot \alpha=\beta$. Hence, $[\alpha]=[\beta]$ and $\bar\phi$ is injective.

\end{proof}

\section{Proofs of the remaining cases}
\begin{Proposition}\label{85}Let $V$ be 8-dimensional. The natural action of $GL(V)$ on $\Lambda^5(V^*)$ has 35 orbits, 31 of those orbits are non-degenerate and three are stable.
\end{Proposition}
\begin{proof}
We fix a volume form $\Omega$ on $V$. In the case at hand, $(n,k)=(8,3)$, both $k$ and $n{-}k$ are odd. Parts (3) and (4) of Lemma \ref{dualityLemma} imply that the map $c$ induces a bijection between the orbits of $\Lambda^5(V^*)$ and the orbits of $\Lambda^3(V)$. This bijection preserves stable orbits. Thus the total number of orbits and the number of stable orbits correspond to the numbers known from the case $(n,k)=(8,3)$. To get the number of non-degenerate orbits one observes that in the case $(n,k)=(7,5)$ there are 4 orbits and applies Lemma \ref{degenerateLemma}.
\end{proof}

\begin{Proposition}\label{74}Let $V$ be 7-dimensional. The natural action of $GL(V)$ on $\Lambda^4(V^*)$ has 20 orbits, 15 of those orbits are non-degenerate and 4 are stable.
\end{Proposition}

\begin{proof}
We fix a volume form $\Omega$ on $V$. Let $\xi_1,...,\xi_{14}$ be representatives of the orbits in $\Lambda^3(V)$, such that $\xi_1,...,\xi_6$ are degenerate and $\xi_{13}$ and $\xi_{14}$ are stable. The degree $k=3$ of the multivectors $\xi$ is odd, so by Lemma \ref{dualityLemma} part (4), we know that the orbits of $c(\xi_i)$ and $c(\xi_j)$ are different for $i\neq j$. Identifying the orbits in $\Lambda^4(V^*)$ thus amounts to finding out for which $i$ the orbits of $c(\xi_i)$ and $-c(\xi_i)$ coincide. Lemma \ref{dualityLemma} part (2) implies that the orbits coincide if and and only if the stabilizer $Stab(\xi_i)\subset GL(V)$ contains elements of negative determinant.\\

Fix $i\in \{1,...,6\}$, i.e. $\xi_i$ is degenerate. Then there exists an $a\in V\backslash\{0\}$ such that $\iota_a\xi_i=0$ and $W\subset V$, a complement of $\mathbb R\cdot a$. The element $g\in GL(V)$ with $g(a)=-a$ and $g|_W=id_W$ has negative determinant and stabilizes $\xi_i$. In the remaining cases $i\in\{7,...,14\}$ the explicit list of stabilizers given in \cite{MR1982435} yields the following:
\begin{itemize}
\item $i\in\{7,...,12\}$: Two of the stabilizers $Stab(\xi_i)$ of non-degenerate non-stable orbits contain elements of negative determinant
\item $i\in\{13,14\}$: The stabilizers $Stab(\xi_i)$ of stable orbits do not contain elements of negative determinant.
\end{itemize}
Consequently there are $8+2\cdot 6=20$ types of orbits, four of which are stable. We get $15$ as the number of non-degenerate orbits by Lemma \ref{degenerateLemma}.
\end{proof}

\bibliographystyle{habbrv}
\bibliography{document}

\end{document}